\documentclass[11pt,twoside]{amsart}
\usepackage{latexsym,amssymb,amsmath}

\numberwithin{equation}{section}
\newtheorem{theorem}{Theorem}[section]
\newtheorem{lemma}[theorem]{Lemma}
\newtheorem{proposition}[theorem]{Proposition}
\newtheorem{corollary}[theorem]{Corollary}

\theoremstyle{definition}
\newtheorem{definition}[theorem]{Definition} 
 
\newtheorem{remark}[theorem]{Remark}
\newtheorem{example}[theorem]{Example}

\begin{document}

%%%%%%%%%%%%%%%%%%%%%%%%%%%%%%%%%%%%%%%%%%%%%%%%%%%%%%%%%%%%%%%%%%%%%

\title[]{Affine cartesian codes} 

\thanks{The second author was supported by 
COFAA-IPN and SNI. The third author was supported by
SNI}

\author{Hiram H. L\'opez}
\address{
Departamento de
Matem\'aticas\\
Centro de Investigaci\'on y de Estudios
Avanzados del
IPN\\
Apartado Postal
14--740 \\
07000 Mexico City, D.F.
}
\email{hlopez@math.cinvestav.mx}

\author{Carlos Renter\'\i a-M\'arquez}
\address{
Departamento de Matem\'aticas\\
Escuela Superior de F\'\i sica y
Matem\'aticas\\
Instituto Polit\'ecnico Nacional\\
07300 Mexico City, D.F.}
\email{renteri@esfm.ipn.mx}

\author{Rafael H. Villarreal}
\address{
Departamento de
Matem\'aticas\\
Centro de Investigaci\'on y de Estudios
Avanzados del
IPN\\
Apartado Postal
14--740 \\
07000 Mexico City, D.F.
}
\email{vila@math.cinvestav.mx}
%\urladdr{http://www.math.cinvestav.mx/$\sim$vila/}

\begin{abstract} 
We compute the basic parameters (dimension, length,
minimum distance) of affine evaluation codes defined on a cartesian
product of finite sets. Given a sequence of positive
integers, we construct an evaluation code, over a degenerate
torus, with prescribed parameters of a certain type. 
As an application of our results, we recover the 
formulas for the minimum distance of various families of 
evaluation codes.
\end{abstract}

\maketitle 

\section{Introduction}\label{intro-cartesian-codes}

Let $K$ be an arbitrary field and let $A_1,\ldots,A_n$ be 
a collection of non-empty subsets of $K$ with a finite 
number of elements. Consider the following finite sets: (a) the  {\it
cartesian product\/} 
$$
X^*:=A_1\times\cdots\times A_n\subset\mathbb{A}^{n}, 
$$
where $\mathbb{A}^{n}=K^n$ is an affine
space over the field $K$, and (b) 
the {\it projective closure\/} of $X^*$
$$
Y:=\{[(\gamma_1,\ldots,\gamma_n,1)]\, \vert\, \gamma_i\in A_i
\mbox{ for all }i\}\subset\mathbb{P}^{n},
$$
where $\mathbb{P}^{n}$ is a projective space over the field $K$. We 
also consider $X$, the image of $X^*\setminus\{0\}$ under
the map $\mathbb{A}^n\setminus\{0\}\mapsto \mathbb{P}^{n-1}$, $\gamma\mapsto
[\gamma]$. In what follows $d_i$ denotes $|A_i|$, the
cardinality of $A_i$ for $i=1,\ldots,n$. We may always 
assume that $2\leq d_i\leq d_{i+1}$ for all $i$ (see
Proposition~\ref{hiram}).  
As usual, we denote the finite field with $q$ elements by
$\mathbb{F}_q$. The multiplicative group of the field $K$ will be denoted by $K^*$. 

Let $S=K[t_1,\ldots,t_n]$ 
be a polynomial ring, let $P_1,\ldots,P_m$ be the points of $X^*$,
and let $S_{\leq d}$ be 
the $K$-vector space of all polynomials of $S$ of
degree at most $d$. The {\it evaluation map\/} 
\begin{equation*}
{\rm ev}_d\colon S_{\leq d}\longrightarrow K^{|X^*|},\ \ \ \ \ 
f\mapsto \left(f(P_1),\ldots,f(P_m)\right),
\end{equation*}
defines a linear map of
$K$-vector spaces. The image of ${\rm ev}_d$, denoted by $C_{X^*}(d)$,
defines a {\it linear code\/}. Permitting an abuse of language, we are
referring to $C_{X^*}(d)$ as a {\it
linear code\/}, even though the field $K$ might not be finite. We call
$C_{X^*}(d)$ the {\it affine cartesian evaluation code\/} ({\it
cartesian code} for short) of
degree $d$ on the set $X^*$. 
If $K$ is finite, cartesian 
codes are special types of affine Reed-Muller codes in the sense of 
\cite[p.~37]{tsfasman}. 

The {\it dimension\/} and the {\it length\/} 
are two of the basic parameters of $C_{X^*}(d)$, 
they are defined as $\dim_K C_{X^*}(d)$ and $|{X^*}|$, respectively. 
A third
basic parameter of $C_{X^*}(d)$ is the {\it minimum
distance}, which is given by 
$$
\delta_{X^*}(d)=\min\{\|{\rm ev}_d(f)\|
\colon {\rm ev}_d(f)\neq 0; f\in S_{\leq d}\},
$$
where $\|{\rm ev}_d(f)\|$ is the number of non-zero
entries of ${\rm ev}_d(f)$. It is well known that the code $C_{X^*}(d)$ has
the same parameters that $C_Y(d)$, the
projective evaluation code of degree $d$ on $Y$. We give a short proof
of this fact by showing that these codes are equal
(Proposition~\ref{bridge-affine-projective}).

The main results of this paper describe the basic parameters of
cartesian evaluation codes and show the existence of cartesian
codes---over degenerate tori---with prescribed parameters of a certain
type.

Some families of evaluation codes---including several 
variations of Reed-Muller codes---have been studied extensively using
commutative algebra methods (e.g., Hilbert functions, resolutions, Gr\"obner
bases), see \cite{delsarte-goethals-macwilliams,
duursma-renteria-tapia,gold-little-schenck,GRT,affine-codes,
algcodes,renteria-tapia-ca,renteria-tapia-ca2,sorensen,tohaneanu}. In this
paper we use these methods to study the family of cartesian codes. 

A key observation that allows us to use commutative algebra methods to
study evaluation codes is that the kernel of the evaluation map ${\rm
ev}_d$ is precisely $S_{\leq d}\cap I(X^*)$, where $I(X^*)$
is the {\it vanishing ideal\/} of $X^*$ consisting of 
all polynomials of $S$ that vanish on ${X^*}$. Thus, as is seen in
the references given above, the algebra of $S/I(X^*)$ is related to
the basic parameters of $C_{X^*}(d)$. Below we will clarify some more
the role of commutative algebra in coding theory.

Let $S[u]=\oplus_{d=0}^\infty
S[u]_d$ be a polynomial ring with the standard grading, where
$u=t_{n+1}$ is a new variable. Recall that the {\it vanishing ideal\/} of
$Y$, denoted by $I(Y)$, is the ideal of $S[u]$ generated by the
homogeneous polynomials that vanish on $Y$.  We use the algebraic
invariants (regularity, 
degree, Hilbert function) of the graded ring $S[u]/I(Y)$ as a tool 
to study the described codes.  It is a 
fact that this graded ring has the same invariants that the 
affine ring $S/I(X^*)$ \cite[Remark~5.3.16]{singular}. The {\it Hilbert
function\/} of $S[u]/I(Y)$ is given by 

$
\ \ \ \ \ \ \ \ \ \ \ \ \ \ \ \ \ \ \ \ \ \ 
\ \ \ \ \ \ \ H_Y(d):=
\dim_K(S[u]_d/I({Y})\cap S[u]_d).
$

According to \cite[Lecture 13]{harris}, 
we have that $H_Y(d)=|Y|$ for
$d\geq |Y|-1$. This means that $|Y|$ is the {\it degree} of $S[u]/I(Y)$ in
the sense of algebraic geometry \cite[p.~166]{harris}. The
\emph{regularity\/} of $S[u]/I(Y)$,  
denoted by ${\rm reg}\, S[u]/I(Y)$, is the least integer $\ell\geq 0$ such that
$H_Y(d)=|Y|$ for $d\geq \ell$. 

The algebraic invariants 
of $S[u]/I(Y)$ occur in algebraic coding 
theory, as we now briefly explain. The code $C_{X^*}(d)$, has
{\it length\/} $|Y|$ and {\it dimension} 
$H_Y(d)$. The knowledge of the 
regularity of $S[u]/I(Y)$ is important for applications to 
coding theory: for $d\geq {\rm reg}\, S[u]/I(Y)$ the code $C_{X^*}(d)$ coincides
with the underlying vector space $K^{|X^*|}$ and has, 
accordingly, minimum distance equal to $1$. Thus, potentially 
good codes $C_{X^*}(d)$ can occur only if $1\leq d < {\rm
reg}(S[u]/I(Y))$.

The contents of this paper are as follows. We show that the vanishing
ideal $I(Y)$ is a complete intersection (Proposition~\ref{ci-summary-i(y)}). Then, 
one can use \cite[Corollary~2.6]{duursma-renteria-tapia} to compute the 
algebraic invariants of $I(Y)$ in terms of the sequence $d_1,\ldots,d_n$. As a 
consequence, we compute the dimension of $C_{X^*}(d)$ and show that
$\delta_{X^*}(d)=1$ for $d\geq \sum_{i=1}^n(d_i-1)$ (Theorem~\ref{dim-length}). 

In Section~\ref{cartesian-section}, we show upper bounds in terms of
$d_1,\ldots,d_n$ on the  
number of roots, over $X^*$, of polynomials in $S$ which do not 
vanish at all points of $X^*$ (Proposition~\ref{maria-vila-bound},
Corollary~\ref{maria-vila-bound-affine}). The 
main theorem of 
Section~\ref{cartesian-section} is a formula for
the minimum distance of $C_{X^*}(d)$
(Theorem~\ref{lopez-renteria-vila}). 
In general, the problem of computing the minimum
distance of a linear code is difficult because it is NP-hard
\cite{vardy-complexity}.  The
basic parameters 
of evaluation codes over finite 
fields have been computed in a number of cases. Our main results 
provide unifying tools to treat some of these cases. As an application, if
$Y$ is a projective
torus in $\mathbb{P}^n$ over a finite field $K$, we recover a formula
of \cite{ci-codes} for the minimum distance of $C_Y(d)$ 
(Corollary~\ref{maria-vila-hiram-eliseo}). If $Y$ is the image of
$\mathbb{A}^n$ under the map $\mathbb{A}^{n}\rightarrow
\mathbb{P}^{n} $, $x\mapsto [(x,1)]$, we also recover a formula of
\cite{delsarte-goethals-macwilliams} for the minimum distance of
$C_Y(d)$ (Corollary~\ref{delsarte-etal}). If $Y=\mathbb{P}^{n}$,  
the parameters of $C_Y(d)$ are described in
\cite[Theorem~1]{sorensen} (see also \cite{lachaud}), 
notice that in this case $Y$ does not
arises as the projective closure of some cartesian product $X^*$. 

Finally, in Section~\ref{degenerate-tori-section}, we consider 
cartesian codes over degenerate tori. Given a sequence
$d_1,\ldots,d_n$ of positive integers, there exists a finite field
$\mathbb{F}_q$ such that $d_i$ divides $q-1$ for all $i$. We use this
field to construct a cartesian code---over a degenerate torus---with
previously fixed parameters, expressed in terms of $d_1,\ldots,d_n$ 
(Theorem~\ref{construction-of-cartesian-codes}). As a byproduct, 
we obtain formulae for the basic parameters of any affine 
evaluation code over a degenerate torus (see
Definition~\ref{degenerate-torus-def}). Thus, we are also 
recovering the main results of \cite{evaluation11,evaluation}
(Remark~\ref{recovering-evaluation}).

It should be mentioned that we do not know of any efficient decoding
algorithm for the family of 
cartesian codes.  The reader is referred to \cite[Chapter~9]{CLO1},
\cite{joyner-decoding,van-lint} and the references there for some
available decoding algorithms for some families of linear codes. 

For all unexplained
terminology and additional information,  we refer to 
\cite{Eisen,harris,Sta1} (for commutative algebra and the
theory of Hilbert functions), and 
\cite{MacWilliams-Sloane,stichtenoth,tsfasman} (for the theory of linear codes).

\section{Complete intersections and algebraic invariants}\label{ci-section}

We keep the same notations and definitions used in
Section~\ref{intro-cartesian-codes}. In what follows $d_i$ denotes $|A_i|$, the
cardinality of $A_i$ for $i=1,\ldots,n$. 
In this section we show that
$I(Y)$ is a complete intersection and compute the algebraic
invariants of $I(Y)$ in terms of $d_1, \ldots,d_n$.

\begin{theorem}{\rm(Combinatorial Nullstellensatz
\cite[Theorem~1.2]{alon-cn})}
\label{comb-null} Let $S=K[t_1,\ldots,t_n]$ be a
polynomial ring over a field $K$, let $f\in S$, and let
$a=(a_i)\in\mathbb{N}^n$. Suppose that the coefficient of
$t^a$ in $f$ is non-zero and $\deg\left(f\right)=a_1+\cdots+a_n$. If
$A_{1},\ldots ,A_{n}$ are subsets of $K$, with $\left|A_{i}\right| > a_i$ for
all $i$, then there are $x_{1}\in A_{1},\ldots,x_n\in A_n$ such that
$f\left(x_{1},\ldots ,x_{n}\right) \neq 0$.  
\end{theorem}

\begin{lemma}\label{dec12-11} $(\mathrm{a})$ $|Y|=|X^*|=d_1\cdots d_n$. 

$(\mathrm{b})$ If $A_i$ is a subgroup of $(K^*,\cdot\,)$ for all $i$, then
$|X^*|/|A_1\cap\cdots\cap A_n|=|X|$.

$(\mathrm{c})$ If $G\in I(X^*)$ and $\deg_{t_{i}}\left(G\right)
< d_i$ for $i=1,\ldots ,n$, then $G=0$.
\end{lemma}
\begin{proof} (a) The map $X^*\mapsto Y$, $x\mapsto [(x,1)]$, is
bijective. Thus, $|Y|=|X^*|$. (b) Since $A_i$ is a group for all $i$,
the sets $X^*$ and $X$ are also groups under componentwise
multiplication. Thus, there is an epimorphism of 
groups $X^*\mapsto X$,
$x\mapsto [x]$, 
whose kernel is equal to 
$$
\{(\gamma,\ldots,\gamma)\in X^*\colon \gamma\in 
A_1\cap\cdots\cap A_n\}.
$$
Thus, $|X^*|/|A_1\cap\cdots\cap A_n|=|X|$. To show (c) we 
proceed by contradiction. Assume that $G$ is
non-zero. Then, there is a monomial $t^a=t_1^{a_1}\cdots t_n^{a_n}$ of $G$ with 
$\deg(G)=a_1+\cdots+a_n$, where $a=(a_1,\ldots,a_n)$ and $a_i>0$ for
some $i$. As $\deg_{t_{i}}(G) <
d_i$ for all $i$, then $a_i<\left|A_{i}\right| =d_i$ for all $i$. Thus, by 
Theorem~\ref{comb-null}, there are $x_{1},\ldots,x_{n}$ with $x_i\in
A_i$ for all $i$ such that 
$G\left( x_{1},\ldots ,x_{n}\right) \neq 0$, a contradiction to the
assumption that $G$ vanishes on $X^*$. 
\end{proof}

\begin{lemma}\label{epifania-dec29-11} Let $f_i$ be the 
polynomial $\prod_{\gamma\in A_i}(t_i-\gamma)$ for $1\leq i\leq n$. Then
$$
I(X^*)=(f_1,\ldots,f_n).
$$
\end{lemma}

\begin{proof}
``$\supset$'' This inclusion is clear because $f_i$ vanishes on $X^*$
by construction. ``$\subset$'' Take $f$ in $I(X^*)$. Let $\succ$ be the reverse
lexicographical order on the monomials of $S$. By the division 
algorithm  \cite[Theorem~1.5.9, p.~30]{AL}, we can write 
$$
f=g_1f_1+\cdots+g_nf_n+G,
$$
where each of the terms of $G$ is not divisible by any of the
leading monomials $t_1^{d_1},\ldots,t_n^{d_n}$, i.e., $\deg_{t_i}(G)<d_i$ 
for all $i$. As $G$ belongs to $I(X^*)$, by Lemma~\ref{dec12-11}, we
get that $G=0$. Thus, $f\in (f_1,\ldots,f_n)$.
\end{proof}

The degree and the regularity of $S[u]/I(Y)$ can be computed from its
Hilbert series. Indeed, the Hilbert series can be written as
$$
F_Y(t):=\sum_{i=0}^{\infty}H_Y(i)t^i=\sum_{i=0}^{\infty}\dim_K(S[u]/I(Y))_it^i=
\frac{h_0+h_1t+\cdots+h_rt^r}{1-t},
$$
where $h_0,\ldots,h_r$ are positive integers. This follows from the
fact that $I(Y)$ is a 
Cohen-Macaulay ideal of height $n$ \cite{geramita-cayley-bacharach}. 
The number $r$ is the
regularity of $S[u]/I(Y)$ and $h_0+\cdots+h_r$ is the degree of
$S[u]/I(Y)$ (see \cite[Corollary~4.1.12]{monalg}). 

\begin{definition}
A homogeneous ideal $I\subset S$ is called a {\it complete intersection\/} if 
there exists homogeneous polynomials $g_1,\ldots,g_{r}$ such that
$I=(g_1,\ldots,g_{r})$,  
where $r$ is the height of $I$. 
\end{definition}

\begin{proposition}\label{ci-summary-i(y)} {\rm (a)}
$I(Y)=(\,\prod_{\gamma\in A_1}(t_1-u\gamma),\ldots,\prod_{\gamma\in
A_n}(t_n-u\gamma))$.

{\rm (b)} $I(Y)$ is a complete intersection. 

{\rm (c)} $F_Y(t)=\prod_{i=1}^n(1+t+\cdots+t^{d_i-1})/(1-t)$.  

{\rm (d)} ${\rm reg}\, S[u]/I(Y)=\sum_{i=1}^{n}(d_i-1)$ and 
${\rm deg}(S[u]/I(Y))=|Y|=d_1\cdots d_n$.
\end{proposition}

\begin{proof} (a) For $i=1,\ldots,n$, we set 
$f_i=\prod_{\gamma\in A_i}(t_i-\gamma)$. Let $\succ$ be the reverse
lexicographical order on the monomials of $S[u]$. Since
$f_1,\ldots,f_n$ form a Gr\"obner basis with respect
to this order, by 
Lemma~\ref{epifania-dec29-11} and \cite[Lemma~3.7]{affine-codes}, the
vanishing ideal $I(Y)$ is equal to 
$(f_1^h,\ldots,f_n^h)$, where $f_i^h=\prod_{\gamma\in A_i}(t_i-u\gamma)$
is the homogenization of $f_i$ with respect to a new variable $u$. Part
(b) follows from (a) because $I(Y)$ is an ideal of height $n$
\cite{geramita-cayley-bacharach}. (c) This part follows using
(a) and a well known formula for the
Hilbert series of a complete intersection (see
\cite[p.~104]{monalg}). 
(d) This part follows directly from \cite[Corollary~2.6]{duursma-renteria-tapia}.
\end{proof}

\begin{definition}\label{def-param-proj-code}
Let $\{Q_i\}_{i=1}^m$ be a set of representatives for the points of
$Y$. The map
$$
{\rm ev}'_d\colon S[u]_d\rightarrow K^{|Y|},\ \ \ \ \ 
f\mapsto
\left({f(Q_i)}/{f_0(Q_i)}\right)_{i=1}^m,
$$
where $f_0(t_1,\ldots,t_{n},u)=u^d$, defines a linear map of 
$K$-vector spaces. The image of ${\rm ev}'_d$, denoted by $C_Y(d)$, is
called a {\it projective evaluation code\/} of
degree $d$ on the set $Y$. 
\end{definition}

It is not hard to see that the map ${\rm ev}'_d$ is
independent of the set of representatives that we choose for the
points of $Y$. 

\begin{definition}
The {\it affine Hilbert function\/} of $S/I(X^*)$
is given by 
$$
H_{X^*}(d):=\dim_K\, S_{\leq d}/I(X^*)_{\leq d},\ \ \mbox{ where }\ \ 
I(X^*)_{\leq d}=S_{\leq d}\cap I(X^*).
$$
\end{definition}

As the evaluation map ${\rm ev}_d$ induces an isomorphism 
$S_{\leq d}/I(X^*)_{\leq d}\simeq C_{X^*}(d)$, as $K$-vector spaces,
the dimension of 
$C_{X^*}(d)$ is $H_{X^*}(d)$. 

\begin{lemma}{\rm\cite[Remark~5.3.16]{singular}}\label{affine-projective-hilbert}
$H_{X^*}(d)=H_Y(d)$ 
for $d\geq 0$. 
\end{lemma}

In particular, from this lemma, the dimension and the length of
the cartesian code $C_{X^*}(d)$ are  $H_Y(d)$ and ${\rm
deg}(S[u]/I(Y))$, respectively.

\begin{proposition}
\label{bridge-affine-projective}
$C_{X^*}(d)=C_Y(d)$ for $d\geq 1$. 
\end{proposition}

\begin{proof} Since $S[u]_d/I(Y)_d \simeq C_Y(d)$ and $S_{\leq
d}/I(X^*)_{\leq d} \simeq C_{X^*}(d)$, by
Lemma~\ref{affine-projective-hilbert}, we get that the linear codes $C_{X^*}(d)$ and
$C_Y(d)$ have the same dimension, and the same length. Thus, it suffices to show the
inclusion ``$\supset$''. Any point of $C_{Y}(d)$ has the form
$W=(f(P_i,1))_{i=1}^m$, 
where $P_1,\ldots,P_m$ are the points of $X^*$ and $f\in S[u]_{d}$.
If $\widetilde{f}$ is the polynomial $f(t_1,\ldots,t_n,1)$, then
$\widetilde{f}$ is in $S_{\leq d}$ and 
$f(P_i,1)=\widetilde{f}(P_i)$ for all $i$. Thus, $W$ is in
$C_{X^*}(d)$, as required.
\end{proof}

\section{Cartesian evaluation codes}\label{cartesian-section}
In this section we compute the basic parameters of cartesian codes and
give some applications. If $d$ is at most $\sum_{i=1}^n(d_i-1)$, we
show an upper bound in terms of 
$d_1,\ldots,d_n$ on the  
number of roots, over $X^*$, of polynomials in
$S_{\leq d}$ which do not vanish at all points of $X^*$. 

We begin by computing some of the basic parameters of $C_{X^*}(d)$,
the cartesian evaluation code of degree $d$ on $X^*$.

\begin{theorem}\label{dim-length} The length of $C_{X^*}(d)$ is
$d_1\cdots d_n$, its minimum distance is $1$ 
for $d\geq \sum_{i=1}^n(d_i-1)$, and its
dimension is 
\begin{eqnarray*}
& & H_{X^*}(d)=\binom{n+d}{d}-\sum\limits_{1\leq i\leq n}
\binom{n+d-d_i}{d-d_i}+\sum\limits_{i<j}
\binom{n+d-(d_i+d_j)}{d-(d_i+d_j)}-\\
&&  \sum\limits_{i<j<k}
\binom{n+d-(d_i+d_j+d_k)}{d-(d_i+d_j+d_k)} + \cdots + (-1)^n
\binom{n+d-(d_1+\cdots +d_n)}{d-(d_1+ \cdots +d_n)}.
\end{eqnarray*}
\end{theorem}

\begin{proof} The length of $C_{X^*}(d)$ is $|X^*|=d_1\cdots d_n$. We
set $r=\sum_{i=1}^n(d_i-1)$. By
Proposition~\ref{ci-summary-i(y)}, the regularity of $S[u]/I(Y)$ is
equal to $r$, i.e., $H_Y(d)=|Y|$ for $d\geq r$. Thus, by
Lemmas~\ref{dec12-11} and \ref{affine-projective-hilbert},
$H_{X^*}(d)=|X^*|$ for $d\geq r$, i.e., $C_{X^*}(d)=K^{|X^*|}$ for
$d\geq r$. Hence $\delta_{X^*}(d)=1$ 
for $d\geq r$. 
By Proposition~\ref{ci-summary-i(y)}, the ideal $I(Y)$
is a complete intersection generated by $n$ homogeneous polynomials
$f_1,\ldots,f_n$ of 
degrees $d_1,\ldots,d_n$. Thus, applying
\cite[Corollary~2.6]{duursma-renteria-tapia} and using the equality 
$H_{X^*}(d)=H_Y(d)$, we obtain the required formula for the dimension.
\end{proof}

\begin{proposition}\label{hiram} If $d_1=1$ and $X'=A_2\times\cdots\times A_n$, 
then $C_{X^*}(d)=C_{X'}(d)$ for $d\geq 1$. 
\end{proposition}

\begin{proof} Let $\alpha$ be the only element of $A_1$ and let $Y'$
be the projective closure of $X'$. Then, by
Proposition~\ref{ci-summary-i(y)}, we get
$$
I(Y)=(t_1-u\alpha,f_2^h,\ldots,f_n^h)\ \mbox{ and }\
I(Y')=(f_2^h,\ldots,f_n^h),
$$
where $f_i^h=\prod_{\gamma\in A_i}(t_i-u\gamma)$ for $i=2,\ldots,n$.
Since $S[u]/I(Y)$ and $K[t_2,\ldots,t_n,u]/I(Y')$ have the
same Hilbert function, we get that the dimension and the length of
$C_{X^*}(d)$ and $C_{X'}(d)$ are the same. Thus, to show the equality
$C_{X^*}(d)=C_{X'}(d)$, it suffices to show the inclusion
``$\subset$''. Any element of $C_{X^*}(d)$ has the form
$$
W=(f(\alpha,Q_1),\ldots,f(\alpha,Q_m)),
$$
where $Q_1,\ldots,Q_m$ are the points of $X'$ and $f\in S_{\leq d}$.
If $\widetilde{f}$ is the polynomial $f(\alpha,t_2,\ldots,t_n)$, then
$\widetilde{f}$ is in $K[t_2,\ldots,t_n]_{\leq d}$ and 
$f(\alpha,Q_i)=\widetilde{f}(Q_i)$ for all $i$. Thus, $W$ is in
$C_{X'}(d)$, as required.
\end{proof}

Since permuting the
sets $A_1,\ldots,A_n$ does not affect neither the parameters of the 
corresponding cartesian evaluation codes, nor the invariants of the
corresponding vanishing ideal, by Proposition~\ref{hiram} we may always 
assume that $2\leq d_i\leq d_{i+1}$ for all $i$, where $d_i=|A_i|$.

For $G\in S$, we denote the zero set of $G$ in $X^*$ by $Z_{X^*}(G)$.
We begin with a general 
bound that will be refined later in this section. 
The proof of \cite[Lemma 3A, p.~147]{schmidt} can be easily adapted
to obtain the following auxiliary result.

\begin{lemma} \label{vila-vaz-pinto-sequence} Let $0\neq
G=G(t_1,\ldots,t_n)\in S$ be a polynomial of total degree $d$. If
$d_i\leq d_{i+1}$ for all $i$, then
$$
|Z_{X^*}(G)|\leq
\left\{\hspace{-1mm}\begin{array}{ll}
d_2\cdots d_nd&\mbox{ if }
n\geq 2,\\
d &\mbox{ if } n=1.
\end{array}
\right.
$$
\end{lemma}

\begin{proof} By induction on $n+d\geq 1$. If $n+d=1$, then $n=1$,
$d=0$ and the result is obvious. If $n=1$, then the result is
clear because $G$ has at most $d$ roots in $K$. Thus, we may assume
$d\geq 1$ and $n\geq 2$. We can write $G$ as
\begin{equation}\label{maria-eq-1}
G=G(t_1,\ldots,t_n)=G_0(t_1,\ldots,t_{n-1})+G_1(t_1,\ldots,t_{n-1})t_n+\cdots+
G_r(t_1,\ldots,t_{n-1})t_n^r,\tag{$\dag$}
\end{equation}
where $G_r\neq 0$ and $0\leq r\leq d$. Let
$\beta_1,\ldots,\beta_{d_1}$ be the elements of $A_1$. We set
$$
H_k=H_k(t_2,\ldots,t_n):=G(\beta_k,t_2,\ldots,t_n)\ \ \mbox{ for }\ \ 1\leq k\leq d_1.
$$

Case (I): $H_k(t_2,\ldots,t_n)=0$ for some $1\leq k\leq d_1$. From
Eq.~(\ref{maria-eq-1}) we get
$$
H_k(t_2,\ldots,t_n)=G_0(\beta_k,t_2,\ldots,t_{n-1})+G_1(\beta_k,t_2,\ldots,t_{n-1})t_n+
\cdots+ G_r(\beta_k,t_2,\ldots,t_{n-1})t_n^r=0.
$$
Therefore $G_i(\beta_k,t_2,\ldots,t_{n-1})=0$ for $i=0,\ldots,r$.
Hence $t_1-\beta_k$ divides $G_i(t_1,\ldots,t_{n-1})$ for all $i$.
Thus, by Eq.~(\ref{maria-eq-1}), we can write
$$
G(t_1,\ldots,t_n)=(t_1-\beta_k)G'(t_1,\ldots,t_n)
$$
for some $G'\in S$. Notice that $\deg(G')+n=d-1+n<d+n$. Hence, by
induction, we get
$$
|Z_{X^*}(G)|\leq |Z_{X^*}{(t_1-\beta_k)}|+|Z_{X^*}(G'(t_1,\ldots,t_n))|\leq
d_2\cdots d_n+d_2\cdots d_n (d-1)=d_2\cdots d_nd.
$$

Case (II): $H_k(t_2,\ldots,t_n)\neq 0$ for $1\leq k\leq d_1$.
Observe the inclusion
$$
Z_{X^*}(G)\subset\bigcup_{k=1}^{d_1}(\{\beta_k\}\times
Z(H_k)),
$$
where $Z(H_k)=\{a\in A_2\times\cdots\times A_n\,
\vert\, H_k(a)=0\}$. As $\deg(H_k)
+n-1<d+n$ and $d_i\leq d_{i+1}$ for all $i$, then by induction 
$$
|Z_{X^*}(G)|\leq\sum_{k=1}^{d_1}|Z(H_k)|\leq
d_1d_3\cdots d_nd\leq d_2d_3\cdots d_nd,
$$
as required.
\end{proof}

\begin{lemma}\label{feb21-10} Let $d_1,\ldots,d_{n-1},d',d$ be
positive integers such that 
$d=\sum_{i=1}^k(d_i-1)+\ell$ and $d'=\sum_{i=1}^{k'}(d_i-1)+\ell'$
for some integers 
$k,k',\ell,\ell'$ satisfying that $0\leq k,k'\leq n-2$ and 
$1\leq\ell\leq d_{k+1}-1$, $1\leq\ell'\leq d_{k'+1}-1$. 
If $d'\leq d$ and $d_i\leq d_{i+1}$ for all $i$, then
$k'\leq k$ and
\begin{equation}\label{jan11-12}
-d_{k'+1}\cdots d_{n-1}+\ell'd_{k'+2}\cdots d_{n-1}\leq -d_{k+1}\cdots
d_{n-1}+\ell d_{k+2}\cdots d_{n-1},\tag{$*$}
\end{equation}
where $d_{k+2}\cdots d_{n-1}=1$
$($resp., $d_{k'+2}\cdots d_{n-1}=1${\rm)} if $k=n-2$ $($resp.,
$k'=n-2${\rm)}. 
\end{lemma}

\begin{proof} First we show that $k'\leq k$. If $k'>k$, from the
equality 
$$
\ell=(d-d')+\ell'+[(d_{k+1}-1)+\cdots+(d_{k'+1}-1)],
$$
we obtain that $\ell\geq d_{k+1}$, a contradiction. Thus, $k'\leq k$.
Since $d_{k+2}\cdots d_{n-1}$ is a common factor of each term of
Eq.~(\ref{jan11-12}), we need only show the equivalent inequality:
\begin{equation}\label{jan11-12-1}
d_{k+1}-\ell\leq (d_{k'+1}-\ell')d_{k'+2}\cdots d_{k+1}.\tag{$**$}
\end{equation}
If $k=k'$, then $d_{k'+2}\cdots d_{k+1}=1$ and $d-d'=\ell-\ell'\geq 0$. Hence,
$\ell\geq \ell'$ and Eq.~(\ref{jan11-12-1}) holds. If $k\geq
k'+1$, then 
$$
d_{k+1}-\ell\leq d_{k+1}\leq d_{k'+2}\cdots d_{k+1}\leq
d_{k'+2}\cdots d_{k+1}(d_{k'+1}-\ell').
$$
Thus, Eq.~(\ref{jan11-12-1}) holds. 
\end{proof}

\begin{lemma}\label{making-sure} If\/ $0\neq G\in S$. Then, there are
$r\geq 0$ distinct elements 
$\beta_1,\ldots,\beta_r$ in $A_n$ and $G'\in S$ such that
$$
G=(t_n-\beta_1)^{a_1}\cdots(t_n-\beta_r)^{a_r}G',\ \ \ \ \ a_i\geq
1\mbox{ for all } i, 
$$
and $G'(t_1,\ldots,t_{n-1},\gamma)\neq 0$ for any $\gamma\in A_n$.
\end{lemma}

\begin{proof} Fix a monomial ordering in $S$.   
If the degree of $G$ is zero, 
we set $r=0$ and $G=G'$. Assume that $\deg(G)>0$. If
$G(t_1,\ldots,t_{n-1},\gamma)\neq 0$ for all 
$\gamma\in A_n$, we set $G=G'$ and $r=0$. If
$G(t_1,\ldots,t_{n-1},\gamma)=0$ 
for some $\gamma\in A_n$, then by the division algorithm there are $F$
and $H$ in $S$ such that $G=(t_n-\gamma)F+H$, where $H$ is a
polynomial whose terms are not divisible by the leading term of
$t_n-\gamma$, i.e., $H$ is a
polynomial in $K[t_1,\ldots,t_{n-1}]$. 
Thus, as $G(t_1,\ldots,t_{n-1},\gamma)=0$,
we get that $H=0$ and $G=(t_n-\gamma)F$. Since $\deg(F)<\deg(G)$, the
result follows using induction on the total degree of $G$.
\end{proof}

\begin{proposition}\label{maria-vila-bound}
Let $G=G(t_1,\ldots,t_n)\in S$ be a polynomial of total degree
$d\geq 1$ such that $\deg_{t_i}(G)\leq d_i-1$ for $i=1,\ldots,n$. If
$d_i\leq d_{i+1}$ for all $i$ and $d=\sum_{i=1}^k(d_i-1)+\ell$ for
some integers $k,\ell$ such that $1\leq \ell\leq d_{k+1}-1$, $0\leq k\leq n-1$,
then
$$
|Z_{X^*}(G)|\leq d_{k+2}\cdots d_n(d_1\cdots d_{k+1}-d_{k+1}+\ell),
$$
where we set $d_{k+2}\cdots d_n=1$ if $k=n-1$.
\end{proposition}

\begin{proof} We proceed by induction on $n$. By
Lemma~\ref{making-sure}, there are $r\geq 0$ distinct elements 
$\beta_1,\ldots,\beta_r$ in 
$A_n$ and $G'\in S$ such that
$$
G=(t_n-\beta_1)^{a_1}\cdots(t_n-\beta_r)^{a_r}G',\ \ \ \ \ a_i\geq
1\mbox{ for all } i, 
$$
and $G'(t_1,\ldots,t_{n-1},\gamma)\neq 0$ for any $\gamma\in A_n$. Notice
that $r\leq \sum_{i=1}^ra_i\leq d_n-1$ because the degree of $G$ in $t_n$ is
at most $d_n-1$. We may assume that
$A_n=\{\beta_1,\ldots,\beta_{d_n}\}$. Let $d_i'$ be the degree  
of $G'(t_1,\ldots,t_{n-1},\beta_i)$ and let $d'=\max\{d_i'\vert\ r+1\leq
i\leq d_n\}$. 

\underline{Case (I)}: Assume $n=1$. Then, $k=0$ and $d=\ell$. Then
$|Z_{X^*}(G)|\leq \ell$ because a non-zero polynomial in one variable
of degree $d$ has at 
most $d$ roots. 

\underline{Case (II)}: Assume $n\geq 2$ and $k=0$. Then, $d=\ell\leq d_1-1$.
Hence, by Lemma~\ref{vila-vaz-pinto-sequence}, we get
$$ 
|Z_{X^*}(G)|\leq d_2\cdots d_nd=d_2\cdots d_n\ell=d_{k+2}\cdots
d_n(d_1\cdots d_{k+1}-d_{k+1}+\ell),  
$$
as required. 

\underline{Case (III)}: Assume $n\geq 2$, $k\geq 1$ and $d'=0$.  Then, 
$|Z_{X^*}(G)|=rd_1\cdots d_{n-1}$. Thus, it suffices to show the 
inequality 
$$
rd_1\cdots d_{n-1}\leq d_{1}\cdots d_n-d_{k+1}\cdots d_{n}+\ell
d_{k+2}\cdots d_{n}.  
$$

All terms of this inequality have $d_{k+2}\cdots d_{n-1}$ as a common
factor. Hence, this case reduces to showing the following
equivalent inequality
$$
rd_1\cdots d_{k+1}\leq d_n(d_1\cdots d_{k+1}-d_{k+1}+\ell).
$$

We can write $d_n=r+1+\delta$ for some $\delta\geq 0$. If we
substitute $d_n$ by $r+1+\delta$, we get the equivalent inequality
$$
d_{k+1}(r+1)\leq \ell r+d_1\cdots d_{k+1}+\ell+\delta
d_1\cdots d_{k+1}-\delta d_{k+1}+\delta\ell .
$$

We can write $d=r+\delta_1$ for some $\delta_1\geq 0$. Next, if we
substitute $r$ by $\sum_{i=1}^{k}(d_i-1)+\ell-\delta_1$ on the left
hand side of this inequality, we get 
$$
0 \leq \ell[r+1+\delta-d_{k+1}]+d_{k+1}[d_1\cdots
d_{k}-1-\textstyle\sum_{i=1}^{k}(d_i-1)+\delta_1]+\delta[d_1\cdots
d_{k+1}-d_{k+1}].
$$
Since $r+1+\delta-d_{k+1}\geq r+1+\delta-d_{n}=0$ and $k\geq 1$, this
inequality 
holds. This completes the proof of this case.

\underline{Case (IV)}: Assume $n\geq 2$, $k\geq 1$ and $d'\geq 1$. We
may assume 
that $\beta_{r+1},\ldots,\beta_m$ are the elements $\beta_i$ of
$\{\beta_{r+1},\ldots,\beta_{d_n}\}$ such that
$G'(t_1,\ldots,t_{n-1},\beta_i)$ has
positive degree. We set 
$$
G_i'=G'(t_1,\ldots,t_{n-1},\beta_i)
$$ 
for $r+1\leq i\leq m$. Notice that $d=\sum_{i=1}^ra_i+\deg(G')\geq
r+d'\geq d_i'$. The
polynomial 
$$H:=(t_n-\beta_1)^{a_1}\cdots(t_n-\beta_r)^{a_r}$$ 
has
exactly $rd_1\cdots d_{n-1}$ roots in $X^*$. Hence, counting the roots
of $G'$ that are not in $Z_{X^*}(H)$, we obtain:
\begin{equation}\label{feb22-10}
|Z_{X^*}(G)|\leq
rd_1\cdots d_{n-1}+\sum_{i=r+1}^{m}|Z(G_i')|,\tag{$\star$}
\end{equation}
where $Z(G_i')$ is the set of zeros of
$G_i'$ in $A_1\times\cdots\times A_{n-1}$. For each $r+1\leq i\leq m$, we can 
write $d_i'=\sum_{i=1}^{k_i'}(d_i-1)+\ell_i'$, with $1\leq
\ell_i'\leq d_{k_i'+1}-1$. The proof of this case will be divided in 
three subcases. 

\underline{Subcase (IV.a)}: Assume $\ell\geq r$ and $k=n-1$. The
degree of $G_i'$ in the 
variable $t_j$ is at most $d_j-1$ for $j=1,\ldots,{n-1}$. Hence, by 
Lemma~\ref{dec12-11}, the non-zero polynomial
$G_i'$ cannot be the
zero-function on $A_1\times\cdots\times A_{n-1}$.  Therefore,
$|Z(G_i')|\leq d_1\cdots d_{n-1}-1$ for $r+1\leq i\leq m$. Thus, by
Eq.~(\ref{feb22-10}), we get the required inequality 
$$
|Z_{X^*}(G)|\leq
rd_1\cdots d_{n-1}+(d_n-r)(d_1\cdots d_{n-1}-1)\leq d_1\cdots
d_n-d_n+\ell,
$$
because in this case $d_{k+2}\cdots d_n=1$ and $\ell\geq r$.

\underline{Subcase (IV.b)}: Assume $\ell>r$ and $k\leq n-2$. Then, we can write
$$
d-r=\sum_{i=1}^k(d_i-1)+(\ell-r)
$$
with $1\leq \ell-r\leq d_{k+1}-1$. Since $d_i'\leq d-r$ for $i=r+1,\ldots,m$, by
applying Lemma~\ref{feb21-10} to the sequence
$d_1,\ldots,d_{n-1},d_i',d-r$, we get $k_i'\leq k$ for $r+1\leq i\leq
m$. By induction hypothesis we can bound $|Z(G_i')|$. Then, using
Eq.~(\ref{feb22-10}) and Lemma~\ref{feb21-10}, we 
obtain: 
\begin{eqnarray*}
|Z_{X^*}(G)|&\leq&
rd_1\cdots d_{n-1}+\sum_{i=r+1}^m d_{k_i'+2}\cdots d_{n-1}(d_1\cdots
d_{k_i'+1}-d_{k_i'+1}+\ell_i')\\ 
&\leq& rd_1\cdots d_{n-1}+(d_n-r)[(d_{k+2}\cdots
d_{n-1})(d_1\cdots d_{k+1}-d_{k+1}+\ell-r)].
\end{eqnarray*}
Thus, by factoring out the common term $d_{k+2}\cdots d_{n-1}$, we
need only show the inequality: 
\begin{eqnarray*}
& & rd_1\cdots d_{k+1}+(d_n-r)(d_1\cdots d_{k+1}-d_{k+1}+\ell-r)\leq \\ 
& &\ \ \ \ \ \ \ \ \ \ \ \ \ \ \ \ \ \ \ \ \ \ \ \ \ \ \ \ \ 
\ \ \ \ \ \ \ \ \ \ \ \ \ \ \ \ \ \ \ \ 
d_{n}(d_1\cdots d_{k+1}-d_{k+1}+\ell).
\end{eqnarray*}

After simplification, we get that this inequality is
equivalent to $r(d_n-d_{k+1}+\ell-r)\geq 0$. This inequality holds
because $d_n\geq d_{k+1}$ and $\ell>r$. 

\underline{Subcase (IV.c)}: Assume $\ell\leq r$. We can write 
$d-r=\sum_{i=1}^s(d_i-1)+\widetilde{\ell}$, where
$1\leq\widetilde{\ell}\leq d_{s+1}-1$ and $s\leq k$. Notice that
$s<k$. Indeed, if $s=k$, then from the equality
\begin{equation}\label{jan21-12}
d-r=\sum_{i=1}^s(d_i-1)+\widetilde{\ell}=\sum_{i=1}^k(d_i-1)+\ell-r\tag{$\star\star$}
\end{equation}
we get that $\widetilde{\ell}=\ell-r\geq 1$, a contradiction. Thus,
$s\leq n-2$. As
$d-r\geq d_i'$,  by applying Lemma~\ref{feb21-10} to
$d_1,\ldots,d_{n-1},d_i',d-r$, we have 
$k_i'\leq s\leq n-2$ for  $i=r+1,\ldots,m$. By induction
hypothesis we can bound $|Z(G_i')|$. Therefore, using Eq.~(\ref{feb22-10}) and
Lemma~\ref{feb21-10}, we
obtain:
\begin{eqnarray*}
|Z_{X^*}(G)|&\leq&
 rd_1\cdots d_{n-1}+\sum_{i=r+1}^m [d_1\cdots d_{n-1}-d_{k_i'+1}\cdots
 d_{n-1}+d_{k_i'+2}\cdots d_{n-1}\ell_i']\\
&\leq&
 rd_1\cdots d_{n-1}+(d_n-r)[d_1\cdots d_{n-1}-d_{s+1}\cdots
 d_{n-1}+d_{s+2}\cdots d_{n-1}\widetilde{\ell}\, ].
\end{eqnarray*}
Thus, we need only show the inequality 
\begin{eqnarray*}
& & rd_1\cdots d_{n-1}+(d_n-r)[d_1\cdots d_{n-1}-d_{s+1}\cdots
 d_{n-1}+d_{s+2}\cdots d_{n-1}\widetilde{\ell}\, ]\leq \\ 
& &\ \ \ \ \ \ \ \ \ \ \ \ \ \ \ \ \ \ \ \ \ \ \ \ \ \ \ \ \ 
\ \ \ \ \ \ \ \ \ \ \ \ \ \ \ \ \ \ \ \ 
d_1\cdots d_{n}-d_{k+1}\cdots
 d_{n}+d_{k+2}\cdots d_{n}\ell.
\end{eqnarray*}
After cancelling out some terms, we get the following equivalent
inequality:
\begin{equation}\label{jan21-12-1}
d_{k+1}\cdots  d_{n}-d_{k+2}\cdots d_{n}\ell\leq  (d_n-r)[d_{s+1}\cdots
 d_{n-1}-d_{s+2}\cdots d_{n-1}\widetilde{\ell}\,
 ].\tag{$\ddag$}
\end{equation}
The proof now reduces to show this inequality. 

\underline{Subcase (IV.c.1)}: Assume $k=n-1$. Then,
Eq.~(\ref{jan21-12-1}) simplifies to 
\begin{equation*}
d_{n}-\ell\leq  (d_n-r)[d_{s+1}\cdots
 d_{n-1}-d_{s+2}\cdots d_{n-1}\widetilde{\ell}\, ].
\end{equation*}
Since $d_n\geq r+1$, it suffices to show the inequality
$$
r+1-\ell\leq  d_{s+2}\cdots d_{n-1}(d_{s+1}-\widetilde{\ell}\, ).
$$
From Eq.~(\ref{jan21-12}), we get
$$
r+(1-\ell)=\ell-\widetilde{\ell}+\sum_{i=s+1}^{n-1}(d_i-1)+(1-\ell)
=-\widetilde{\ell}+d_{s+1}+\sum_{i=s+2}^{n-1}(d_i-1).
$$
Hence, the last inequality is equivalent to
$$
\sum_{i=s+2}^{n-1}(d_i-1)\leq (d_{s+2}\cdots
d_{n-1}-1)(d_{s+1}-\widetilde{\ell}).
$$
This inequality holds because $d_{s+2}\cdots d_{n-1}\geq
\sum_{i=s+2}^{n-1}(d_i-1)+1$. 

\underline{Subcase (IV.c.2)}: Assume $k\leq n-2$. By canceling out the common
term $d_{k+2}\cdots d_{n-1}$ in Eq.~(\ref{jan21-12-1}), we obtain the
following equivalent inequality
$$
d_{k+1}d_{n}-d_{n}\ell\leq  (d_n-r)(d_{s+2}\cdots 
 d_{k+1})(d_{s+1}-\widetilde{\ell}\, ).
$$
We rewrite this inequality as
$$
r(d_{s+2}\cdots 
 d_{k+1})(d_{s+1}-\widetilde{\ell}\, )\leq d_n[(d_{s+2}\cdots 
 d_{k+1})(d_{s+1}-\widetilde{\ell}\, )-d_{k+1}]+\ell d_n.
$$
Since $d_n\geq r+1$ it suffices to show the inequality
\begin{eqnarray*}
& & r(d_{s+2}\cdots 
 d_{k+1})(d_{s+1}-\widetilde{\ell}\, ) \leq \\ 
& &\ \ \ \ \ \ \ \ \  
r[(d_{s+2}\cdots 
 d_{k+1})(d_{s+1}-\widetilde{\ell}\, )-d_{k+1}]+[(d_{s+2}\cdots 
 d_{k+1})(d_{s+1}-\widetilde{\ell}\, )-d_{k+1}]+\ell d_n.
\end{eqnarray*}
After a quick simplification, this inequality reduces to
$$
(r+1)d_{k+1}\leq (d_{s+2}\cdots 
 d_{k+1})(d_{s+1}-\widetilde{\ell}\, )+\ell d_n.
$$
From Eq.~(\ref{jan21-12}), we get 
$r+1=(-\widetilde{\ell}+d_{s+1})+ (\ell+\sum_{i=s+2}^{k}(d_i-1))$.
Hence, the last inequality is equivalent to 
$$
d_{k+1}\sum_{i=s+2}^{k}(d_i-1)\leq d_{k+1}(d_{s+2}\cdots 
 d_k-1)(d_{s+1}-\widetilde{\ell}\, )+\ell (d_n-d_{k+1}).
$$
This inequality holds because $d_{s+2}\cdots d_{k}\geq
\sum_{i=s+2}^{k}(d_i-1)+1$. This completes the proof of 
the proposition. 
\end{proof}

\begin{corollary}\label{maria-vila-bound-affine} Let $d\geq 1$ be an
integer. If $d_i\leq d_{i+1}$ for all $i$ and
$d=\sum_{i=1}^k(d_i-1)+\ell$ for some integers $k,\ell$ 
such that $1\leq \ell\leq d_{k+1}-1$ and $0\leq k\leq n-1$,
then
$$
\max\{|Z_{X^*}(F)|\colon F\in
S_{\leq d};\, F\not\equiv 0\}
\leq d_{k+2}\cdots d_n(d_1\cdots d_{k+1}-d_{k+1}+\ell).
$$
\end{corollary}

\begin{proof} Let $F=F(t_1,\ldots,t_n)\in S$ be an arbitrary polynomial of total
degree $d'\leq d$ such that $F(P)\neq 0$ for some $P\in X^*$. We can
write $d'=\sum_{i=1}^{k'}(d_i-1)+\ell'$ with $1\leq \ell'\leq
d_{k'+1}-1$ and $0\leq k'\leq k$. Let $\prec$ be the
graded reverse lexicographical order on 
the monomials of $S$. In this order $t_1\succ\cdots\succ t_n$. For
$1\leq i\leq n$, let $f_i$ be the polynomial $\prod_{\gamma\in
A_i}(t_i-\gamma)$. Recall that $d_i=|A_i|$, i.e., $f_i$ has degree
$d_i$. 
By the division algorithm \cite[Theorem~1.5.9, 
p.~30]{AL}, we can write
\begin{equation}\label{march12-10}
F=h_1f_1+\cdots+h_{n}f_n+G',\tag{$\dag\dag$}
\end{equation}
for some $G'\in S$ with $\deg_{t_i}(G')\leq d_i-1$ for
$i=1,\ldots,n$ and $\deg(G')=d''\leq d'$. If $G'$ is a constant, 
by Eq.~(\ref{march12-10}) and using that $0\neq F(P)=G'(P)$, we get 
$Z_{X^*}(F)=\emptyset$. Thus, we may assume that the polynomial $G'$
has positive 
degree $d''$. 
We can write
$d''=\sum_{i=1}^{k''}(d_i-1)+\ell''$, where $1\leq \ell''\leq d_{k''+1}$
and $0\leq k''\leq k'$. Notice that $Z_{X^*}(F)=Z_{X^*}(G')$. By
Proposition~\ref{maria-vila-bound}, and applying Lemma~\ref{feb21-10}
to the sequences $d_1,\ldots,d_n,d'',d'$ 
and $d_1,\ldots,d_n,d',d$, we obtain
\begin{eqnarray*}
|Z_{X^*}(F)|=|Z_{X^*}(G')|&\leq& d_1\cdots d_n-d_{k''+1}\cdots d_n
+d_{k''+2}\cdots d_n\ell'' \\
&\leq & d_1\cdots d_n-d_{k'+1}\cdots d_n
+d_{k'+2}\cdots d_n\ell'\\
&\leq & d_1\cdots d_n-d_{k+1}\cdots d_n
+d_{k+2}\cdots d_n\ell.
\end{eqnarray*}
Thus, $|Z_{X^*}(F)|\leq d_1\cdots d_n-d_{k+1}\cdots d_n
+d_{k+2}\cdots d_n\ell$, as required.
\end{proof}

We come to the main result of this section.

\begin{theorem}\label{lopez-renteria-vila}  
Let $K$ be a field and let $C_{X^*}(d)$ be the cartesian evaluation
code of degree $d$ on the 
finite set $X^*=A_1\times\cdots\times A_n\subset K^n$. If $2\leq
d_i\leq d_{i+1}$ for all $i$, with $d_i=|A_i|$, and $d\geq 1$, then
the minimum distance of 
$C_{X^*}(d)$ is given by 
$$
\delta_{X^*}(d)=\left\{\hspace{-1mm}
\begin{array}{ll}\left(d_{k+1}-\ell\right)d_{k+2}\cdots d_n&\mbox{ if }
d\leq \sum\limits_{i=1}^{n}\left(d_i-1\right)-1,\\
\qquad \qquad 1&\mbox{ if } d\geq \sum\limits_{i=1}^{n}\left(d_i-1\right),
\end{array}
\right.
$$
where $k\geq 0$, $\ell$ are the unique integers such that 
$d=\sum_{i=1}^{k}\left(d_i-1\right)+\ell$ and $1\leq \ell \leq
d_{k+1}-1$. 
\end{theorem}

\begin{proof} If $d\geq \sum_{i=1}^n(d_i-1)$, then the minimum
distance of $C_{X^*}(d)$ is equal to $1$ by Theorem~\ref{dim-length}.
Assume that $1\leq d\leq 
\textstyle\sum_{i=1}^{n}\left(d_i-1\right)-1$. We can write
$$
A_i=\{\beta_{i,1},\beta_{i,2},\ldots,\beta_{i,d_i}\}, \ \ \ \
i=1,\ldots,n.
$$
For $1\leq i\leq k+1$, consider the polynomials
$$
f_i=\left\{\hspace{-1mm}
\begin{array}{ll}(\beta_{i,1}-t_i)(\beta_{i,2}-t_i)\cdots
(\beta_{i,d_{i}-1}-t_i)&\mbox{ if }
1\leq i\leq k,\\
(\beta_{k+1,1}-t_{k+1})(\beta_{k+1,2}-t_{k+1})\cdots
(\beta_{k+1,\ell}-t_{k+1})&\mbox{ if } i=k+1.
\end{array}
\right.
$$
The polynomial $G=f_1\cdots f_{k+1}$ has degree $d$ and
$G(\beta_{1,d_1},\beta_{2,d_2},\ldots,\beta_{n,d_n})\neq 0$. From the
equality
\begin{eqnarray*}
Z_{X^*}(G)& =& [(A_1\setminus\{\beta_{1,d_1}\})\times A_2
\times\cdots\times A_n]\cup\\
& &[\{\beta_{1,d_1}\}\times (A_2\setminus\{\beta_{2,d_2}\})\times A_3
\times\cdots\times A_n]\cup\\
& & \ \ \ \ \ \  \ \ \ \ \ \  \ \ \ \ \ \  \ \ \ \ \ \  \ \ \ \ \ \ \vdots\\
& & [\{\beta_{1,d_1}\}\times\cdots\times \{\beta_{k-1,d_{k-1}}\}\times
(A_k\setminus\{\beta_{k,d_k}\})\times A_{k+1}
\times\cdots\times A_n]\cup\\
& &[\{\beta_{1,d_1}\}\times\cdots\times \{\beta_{k,d_{k}}\}\times
\{\beta_{k+1,1},\ldots,\beta_{k+1,\ell}\}\times A_{k+2}
\times\cdots\times A_n],
\end{eqnarray*}
we get that the number of
zeros of $G$ in $X^*$ is given by:
$$
|Z_{X^*}(G)|=\sum_{i=1}^k(d_i-1)(d_{i+1}\cdots d_n)+\ell
d_{k+2}\cdots d_n=d_1\cdots d_n-d_{k+1}\cdots d_n+\ell d_{k+2}\cdots
d_n.
$$
By Lemma~\ref{dec12-11}, one has $|X^*|=d_1\cdots d_n$. Therefore
\begin{eqnarray*}
\delta_{X^*}(d)&=&\min\{\|{\rm ev}_d(F)\|
\colon {\rm ev}_d(F)\neq 0; F\in S_{\leq d}\}=|X|-\max\{|Z_{X^*}(F)|\colon F\in
S_{\leq d};\, F\not\equiv 0\}\\ 
&\leq& d_1\cdots d_n-|Z_{X^*}(G)|=\left(d_{k+1}-\ell\right)d_{k+2}\cdots d_n,
\end{eqnarray*}
where $\|{\rm ev}_d(F)\|$ is the number of non-zero
entries of ${\rm ev}_d(F)$ and $F\not\equiv 0$ means that $F$ is not 
the zero function on $X^*$. Thus
$$
\delta_{X^*}(d)\leq 
(d_{k+1}-\ell)d_{k+2}\cdots d_n.
$$
The reverse inequality follows at once from Corollary~\ref{maria-vila-bound-affine}.
\end{proof}

\begin{definition}\label{projectivetorus-def} If $K$ is a finite
field, the set $\mathbb{T}=\{[(x_1,\ldots,x_{n+1})]\in\mathbb{P}^{n}\vert\, x_i\in
K^*\mbox{ for all }i\}$ is called a {\it projective torus} 
in $\mathbb{P}^{n}$, where $K^*=K\setminus\{0\}$. 
\end{definition}

As a consequence of our main result, we recover the following formula
for the minimum 
distance of a parameterized code over a projective torus.

\begin{corollary}{\cite[Theorem~3.5]{ci-codes}}\label{maria-vila-hiram-eliseo}
Let $K=\mathbb{F}_q$ be a finite field with $q\neq 2$ elements. 
If $\mathbb{T}$ is a projective torus in $\mathbb{P}^{n}$ and $d\geq 1$, 
 then the minimum distance of $C_\mathbb{T}(d)$ is given by 
$$
\delta_\mathbb{T}(d)=\left\{\begin{array}{cll}
(q-1)^{n-k-1}(q-1-\ell)&\mbox{if}&d\leq (q-2)n-1,\\
1&\mbox{if}&d\geq (q-2)n,
\end{array}
 \right.
$$
where $k$ and $\ell$ are the unique integers such that $k\geq 0$,
$1\leq \ell\leq q-2$ and $d=k(q-2)+\ell$. 
\end{corollary}

\begin{proof} If $A_i=K^*$ for $i=1,\ldots,n$, then $X^*=(K^*)^n$, 
$Y=\mathbb{T}$, and $d_i=q-1$ for all $i$. Since
$\delta_{X^*}(d)=\delta_Y(d)$, the result follows at once from 
Theorem~\ref{lopez-renteria-vila}.
\end{proof}

As another consequence of our main result, we recover a formula for the 
minimum distance of an evaluation code over an affine space. 

\begin{corollary}{\rm
\cite[Theorem~2.6.2]{delsarte-goethals-macwilliams}}\label{delsarte-etal} 
Let $K=\mathbb{F}_q$ be a finite field and let $Y$ be the image of
$\mathbb{A}^n$ under the map $\mathbb{A}^{n}\rightarrow
\mathbb{P}^{n} $, $x\mapsto [(x,1)]$. If $d\geq 1$, the minimum
distance of $C_Y(d)$ is given by: 
$$
\delta_Y(d)=\left\{\begin{array}{cll}
(q-\ell)q^{n-k-1}&\mbox{if}&d\leq n(q-1)-1,\\
1&\mbox{if}&d\geq n(q-1),
\end{array}
 \right.
$$
where $k$ and $\ell$ are the unique integers such that $k\geq 0$,
$1\leq \ell\leq q-1$ and $d=k(q-1)+\ell$. 
\end{corollary}

\begin{proof} If $A_i=K$ for $i=1,\ldots,n$, then
$X^*=K^n=\mathbb{A}^n$ and $d_i=q$ 
for all $i$. Since
$\delta_{X^*}(d)=\delta_Y(d)$, the result follows at once from 
Theorem~\ref{lopez-renteria-vila}.
\end{proof}

\begin{example} If $X^*=\mathbb{F}_2^n$, then the basic parameters of
$C_{X^*}(d)$ are given by
\begin{eqnarray*}
|X^*|=2^n, & \dim
C_{X^*}(d)=\sum_{i=0}^d\binom{n}{i},&\delta_{X^*}(d)=2^{n-d},
\ \ \
1\leq d\leq n. 
\end{eqnarray*}
\end{example}

\begin{example} Let $K=\mathbb{F}_9$ be a field with $9$ elements.
Assume that $A_i=K$ 
for $i=1,\ldots,4$. For certain values of $d$, the basic
parameters of $C_{X^*}(d)$ 
are given in the following table:
\begin{eqnarray*}
&&\left.
\begin{array}{c|c|c|c|c|c|c|c|c|c|c|c}
 d & 1 & 2 & 3 & 4 & 5 & 10&16&20&28&31&32\\
   \hline
 |X^*| &  6561 & 6561 & 6561 & 6561 & 6561 & 6561&6561&6561&6561&6561&6561\\ 
   \hline
 \dim C_{X^*}(d)    \   & 5 & 15&35 &70 &126 &
 981&3525&5256&6526&6560 &6561\\ 
   \hline
 \delta_{X^*}(d) & 5832 & 5103 & 4374 & 3645 &2916 &567&81&45&5&2 &1\\ 
\end{array}
\right.
\end{eqnarray*}
\end{example}

\section{Cartesian codes over degenerate
tori}\label{degenerate-tori-section}

Given a non decreasing sequence of positive integers
$d_1,\ldots,d_n$, we construct a cartesian code, over a degenerate
torus, with prescribed
parameters in terms of $d_1,\ldots,d_n$.

\begin{definition}\label{degenerate-torus-def} Let $K=\mathbb{F}_q$
be a finite field and let  
$v=(v_1,\ldots,v_n)$ be a sequence of positive
integers. The set  
$$
X^*=\{(x_1^{v_{1}},\ldots,x_n^{v_n})\, \vert\, x_i\in K^*\mbox{ for
all }i\}\subset\mathbb{A}^{n},
$$
is called a {\it degenerate torus\/} of type $v$.
\end{definition}

The main result of this section is:

\begin{theorem}\label{construction-of-cartesian-codes} 
Let $2\leq d_1\leq \cdots\leq\ d_n$ be a sequence of
integers. Then, there is a finite field $K=\mathbb{F}_q$ and a
degenerate torus $X^*$ such that the length of $C_{X^*}(d)$ 
is $d_1\cdots d_n$, its dimension is 
\begin{eqnarray*}
& & \dim_K\, C_{X^*}(d)=\binom{n+d}{d}-\sum\limits_{1\leq i\leq n}
\binom{n+d-d_i}{d-d_i}+\sum\limits_{i<j}
\binom{n+d-(d_i+d_j)}{d-(d_i+d_j)}-\\
&&  \sum\limits_{i<j<k}
\binom{n+d-(d_i+d_j+d_k)}{d-(d_i+d_j+d_k)} + \cdots + (-1)^n
\binom{n+d-(d_1+\cdots +d_n)}{d-(d_1+ \cdots +d_n)},
\end{eqnarray*}
its minimum distance is $1$ if $d\geq
\sum_{i=1}^{n}(d_i-1)$, and
$$
\delta_{X^*}(d)=(d_{k+1}-\ell)d_{k+2}\cdots d_n\ \ \mbox{ if } \ \ 
\textstyle d\leq \sum_{i=1}^{n}\left(d_i-1\right)-1,
$$
where $k\geq 0$, $\ell$ are the unique integers such that 
$d=\sum_{i=1}^{k}\left(d_i-1\right)+\ell$ and $1\leq \ell \leq
d_{k+1}-1$. 
\end{theorem}

\begin{proof} Pick a prime number $p$ relatively prime
to $m=d_1\cdots d_n$. Then, by Euler formula, $p^{\varphi(m)}\equiv 1\
({\rm mod}\ m)$, where $\varphi$ is the Euler function. 
We set $q=p^{\varphi(m)}$. Hence, there exists a finite field
$\mathbb{F}_q$ with $q$ elements such that $d_i$ divides $q-1$ for
$i=1,\ldots,n$. We set $K=\mathbb{F}_q$.   

Let $\beta$ be a generator of the cyclic group
$(K^*,\, \cdot\, )$. There are positive integers $v_1,\ldots,v_n$ 
such that $q-1=v_id_i$ for $i=1,\ldots,n$. Notice that
$d_i$ is equal to $o(\beta^{v_i})$, the order of $\beta^{v_i}$ for
$i=1,\ldots,n$. We set $A_i=\langle\beta^{v_i}\rangle$, where
$\langle\beta^{v_i}\rangle$ is the subgroup of $K^*$ generated by
$\beta^{v_i}$. If $X^*$ is the cartesian product of $A_1,\ldots,A_n$,
it not hard to see that $X^*$ is given by
$$
X^*=\{(x_1^{v_{1}},\ldots,x_n^{v_n})\, \vert\, x_i\in K^*\mbox{ for all
}i\}\subset\mathbb{A}^{n},
$$
i.e., $X^*$ is a degenerate torus of type $v=(v_1,\ldots,v_n)$. The
length of $|X^*|$ is 
$d_1\cdots d_n$ because $|A_i|=d_i$ for all $i$. The
formulae for the dimension and the minimum distance of $C_{X^*}(d)$
follow from Theorems~\ref{dim-length} and \ref{lopez-renteria-vila}.
\end{proof}

\begin{remark}\label{recovering-evaluation} 
Let $K=\mathbb{F}_q$ be a finite field and let $\beta$ be a
generator of the cyclic group $(K^*,\cdot\, )$. If $X^*$ is a degenerate
torus of type $v=(v_1,\ldots,v_n)$, then $X^*$ is the cartesian
product of $A_1,\ldots,A_n$, where $A_i$ is the cyclic group generated
by $\beta^{v_i}$. Thus, if $d_i=|A_i|$ for $i=1,\ldots,n$, the
affine evaluation code over $X^*$ is a cartesian code. Hence,
according to Theorem~\ref{dim-length} and \ref{lopez-renteria-vila}, the
basic parameters of $C_{X^*}(d)$ can be computed in terms of
$d_1,\ldots,d_n$ as in 
Theorem~\ref{construction-of-cartesian-codes}. Therefore, we are
 recovering the main results of
\cite{evaluation11,evaluation}.
\end{remark}

As an illustration of Theorem~\ref{construction-of-cartesian-codes}
consider the following example. 

\begin{example} Consider the sequence $d_1=2$, $d_2=5$, $d_3=9$. The
prime number $q=181$ satisfies that $d_i$ divides $q-1$ for all $i$.
In this case $v_1=90$, $v_2=36$, $v_3=20$. 
The basic parameters of the cartesian codes $C_{X^*}(d)$, over the degenerate torus 
$$
X^*=\{(x_1^{90},x_2^{36},x_3^{20})\vert\, x_i\in
\mathbb{F}_{181}^*\ \mbox{ for }i=1,2,3\},
$$
are shown in the following table. Notice that the regularity of
$S[u]/I(Y)$ is $13$. 
\begin{eqnarray*}
&&\left.
\begin{array}{c|c|c|c|c|c|c|c|c|c|c|c|c|c}
 d & 1 & 2 & 3 & 4 & 5 & 6 & 7 & 8 & 9 & 10 & 11 & 12 & 13 \\
   \hline
 |X^*| & 90 & 90 & 90 & 90 & 90 & 90 & 90 & 90 & 90 & 90 & 90 & 90 & 90
 \\ 
   \hline
 \dim C_{X^*}(d)    \    & 4 & 9   & 16& 25 &35& 45& 55 & 65 & 74&
 81 & 86 & 89 & 90 \\ 
   \hline
 \delta_{X^*}(d) & 45 & 36 & 27 & 18 & 9 & 8 & 7 & 6 & 5 & 4  & 3  & 2   &
  1\\ 
\end{array}
\right.
\end{eqnarray*}

Notice that if $K'=\mathbb{F}_9$, and we pick subsets $A_1,A_2,A_3$ of
$K'$ with $|A_1|=2$, $|A_2|=5$, $|A_3|=9$, the cartesian evaluation
code $C_{X'}(d)$, over the set $X'=A_1\times A_2\times A_3$, has 
the same parameters that $C_{X^*}(d)$ for any $d\geq 1$.
\end{example}

\noindent {\bf Acknowledgments.} We thank the referees for their
careful reading of the paper and for the improvements that
they suggested.

\bibliographystyle{plain}

\end{document}